\documentclass{amsart}
\usepackage{amsmath, amstext, amsbsy, amssymb}
\usepackage{graphicx}
\usepackage{subfigure}
\usepackage{picinpar}
\usepackage{psfrag}
\input xy \xyoption{all}

\hoffset \voffset \oddsidemargin=55pt \evensidemargin=55pt
\numberwithin{equation}{section}
\def\beq{\begin{eqnarray}}
\def\eeq{\end{eqnarray}}
\def\beqs{\begin{eqnarray*}}
\def\eeqs{\end{eqnarray*}}

\def\mz{{\mathbb Z}}
\def\mc{{\mathbb C}}

\newfont{\df}{eufm10}

\def\mc{{\mathbb C}}

\def\id{\hbox{\rm id}}

\def\sla{\mathfrak{sl}_2}

\title[]
{Simple modules over quantum torus and quantum group}

\author[L. Xia]{Limeng Xia$^{\ast}$}
\address{XIA: Institute of Applied System Analysis, Jiangsu University, XueFu Road 301,
Zhenjiang 212013, Jiangsu,  PR China} \email{xialimeng@ujs.edu.cn}

\author[N. Hu]{Naihong Hu$^{\ast\ast}$}
\address{HU: Department of Mathematics,  Shanghai Key Laboratory of Pure Mathematics and Mathematical Practice,
East China Normal University, Minhang Campus, Dong Chuan Road 500,
Shanghai 200241, PR China} \email{nhhu@math.ecnu.edu.cn}
\thanks{$^{\ast}$ L.M. Xia, supported by the NNSFC (Grant No. 11871249) and the Natural Science Foundation of Jiangsu Province (No. BK20171294).}
\thanks{$^{\ast\ast}$ N.H. Hu, supported by the NNSFC (Grant No.  11771142).}

\date{}
\begin{document}
\maketitle

\newtheorem{theo}{Theorem}[section]
\newtheorem{theorem}[theo]{Theorem}
\newtheorem{defi}[theo]{Definition}
\newtheorem{lemma}[theo]{Lemma}
\newtheorem{coro}[theo]{Corollary}
\newtheorem{prop}[theo]{Proposition}
\newtheorem{remark}[theo]{Remark}

\begin{abstract}  In this paper, we classify all simple modules over the quantum torus $\mc_\nu[x^{\pm1},y^{\pm1}]$ and the quantum group $U_q(\sla)$ for generic case.
\end{abstract}



\smallskip\bigskip

\section{Introduction}

Quantum group $U_q(\sla)$ is the $q$-deformation of the universal
enveloping algebra $U(\sla)$ of the $3$-dimensional simple Lie
algebra $\sla$.  In some sense, $\sla$ and $U_q(\sla)$ perhaps are
the most fundamental objects in the theory of Lie algebras and
quantum groups. The classification of simple modules over $\sla$ or
$U_q(\sla)$ is a very important problem in their representation
theory.

Let ${\bf z}$ be the Casimir element of  $U(\sla)$. The quotient
algebra $U(\sla)/\langle{\bf z}-c\rangle$ is isomorphic to a
subalgebra of Weyl algerba $\mathfrak{A}$ for any $c\in\mc$. In
1981, Block completely classified all simple modules over the Weyl
algerba $\mathfrak{A}$ and the Lie algebra $\sla$ (see \cite{B}).

Let $Z(U_q(\sla))$ be the center of quantum group $U_q(\sla)$. If
$q$ is a root of unity, then $K^{n}, E^n, F^n\in Z(U_q(\sla))$ for some positive integer $n$ and the quotient $U_q(\sla)/\langle K^{n}-c_1, E^n-c_2, F^n-c_3\rangle$ is a finite-dimensional algebra for all $c_1\in\mc^*$ and $c_2, c_3\in\mc$. In this case,  all simple modules have been
determined (See \cite{CK}). However, it is
still open to classify all simple modules over quantum group
$U_q(\sla)$ for generic $q\in\mc^*$.

The quantum torus $\mc_\nu[x^{\pm1},y^{\pm1}]$ is the quantum
analogue of the Weyl algebra $\mathfrak{A}$, which arises as a
localization of some group algebra (see \cite{Br}) and plays an
important role in noncommutative geometry (see [M]). If $\nu$ is a
root of unity of order $n$, the center algebra of the quantum
torus is generated by $x^n, y^n$, and the quotient algebra
$\mc_\nu[x^{\pm1},y^{\pm1}]/\langle x^n-c_1, y^n-c_2\rangle$ is isomorphic to
$\mathfrak{gl}_n(\mc)$ for all $c_1, c_2\in\mc^*$. It is also an open problem to classify  all
simple modules over the quantum torus for generic $\nu\in\mc^*$.

Similar to the classical case,  let $Z_q$ be the Casimir element of
$U_q(\sla)$, then the quotient algebra $U_q(\sla)/\langle
Z_q-c\rangle$ is isomorphic to a subalgebra of
 $\mc_\nu[x^{\pm1},y^{\pm1}]$ for any $c\in\mc$.

In this paper, we classify all simple modules over the quantum torus $\mc_\nu[x^{\pm1},y^{\pm1}]$ and the quantum group $U_q(\sla)$ for generic $\nu=q^2$.

\section{Localizations}

In this section, we first recall some definitions and facts about localizations of noncommutative rings, then we list some known results about
simple modules over noncommutative rings and their localizations (see \cite{B} and \cite{BG}).

Let $R$ be a ring with $1$ and $S$ a multiplicative subset of $R$ containing $1$. We say that $S$ satisfies the left Ore condition if
\beqs Rs\cap Sa\not=\emptyset,\;\forall (s,a)\in S\times R.\eeqs

A localization of $R$ with respect to $S$ is a ring $B=S^{-1}R$
containing $R$ as a subring such that every $s\in S$ is invertible
and $B=\{s^{-1}a\mid s\in S, a\in R\}$. The localization $B=S^{-1}R$
exists if and only if $S$ has no zero divisor and satisfies the left
Ore condition.

%
\begin{lemma}[Lemma 2.4.2 of \cite{B}]\label{Soc}
Suppose that $B$ is a localization of $R$ with respect to $S$ and a
principal left ideal domain which is not a division ring, $M$ is a
simple $S$-torsion-free $R$-module, $\alpha\in R$ is irreducible in
$B$ and annihilates some nonzero element of $M$. Then $M\cong
Soc_RB/B\alpha$.
\end{lemma}

\begin{remark}Suppose that  $\alpha\in R$ is irreducible in $B$, then $B/B\alpha$ is a simple $B$-module. However, the $R$-module $Soc_RB/B\alpha=(R+B\alpha)/B\alpha$ may be not simple. The Lemma 5.2  gives  examples for this case.
\end{remark}

\section{Quantum torus $\mc_\nu[x^{\pm1},y^{\pm1}]$}

Let $\nu\in\mc^*$ be generic and $R=\mc_\nu[x^{\pm1},y^{\pm1}]$ with
the defining relation $\nu xy=yx$. In this section, we classify all
simple $R$-modules.

Let $S=\mc[x, x^{-1}]\setminus\{0\}$, which is multiplicative, contains $1$ and has no zero divisor.
\begin{lemma}\label{Ore}
The subset $S$ satisfies the left Ore condition.
\end{lemma}
\begin{proof}
If $a=0$, we have $0\in Sa\cap Rs$ for all $s\in S$.

For all $s=s(x)\in\mc[x, x^{-1}]$ and $a=\sum f_i(x)y^i\in R\setminus\{0\}$, choose
\beqs h(x)=\prod_{f_i\not=0}s(\nu^ix),\;b_i(x)=\frac{h(x)f_i(x)}{s(\nu^{i}x)}\in\mc[x, x^{-1}].\eeqs
Then we have
\beqs R\cdot s(x)\ni(\sum b_i(x)y^i)s(x)=\sum h(x)f_i(x)y^i=h(x)a\in S\cdot a. \eeqs
So $S$ satisfies the left Ore condition.
\end{proof}

The localization of $R$ with respect to $S$ is $B=S^{-1}R=\mc(x)[y,
y^{-1}]$. In particular, $yf(x)=f(\nu x)y$, for all $f(x)\in\mc(x)$.
\begin{lemma}
The ring $B=\mc(x)[y, y^{-1}]$ is a principal left ideal domain and is not a division ring.
\end{lemma}
\begin{proof}
In fact, $\alpha\in B$ is invertible if and only if it is a non-zero monomial in variable $y$. So $B$ is not a division ring.
Moreover, $B$ is Euclidean, and then it is a principal left ideal domain.
\end{proof}

\begin{lemma}
For any $\lambda\in\mc^*$, let $M_\lambda$ be the vector space $\mc[y, y^{-1}]{\bf 1}_\lambda$. Then $M_\lambda$ is a
simple module over $R$ with $x\cdot{\bf1}_\lambda=\lambda{\bf1}_\lambda$ and $y$ acts by multipling.
\end{lemma}
\begin{proof}
For all $v=(\sum_{i=k}^l c_iy^i){\bf1}_\lambda\not=0$, we have
\beqs \lambda^{-n}x^n\cdot v=(\sum_{i=k}^l \nu^{-in}c_iy^i){\bf1}_\lambda.\eeqs
Since $\nu$ is generic, the matrix $(q^{-in})_{k\leq i\leq l, 0\leq n\leq l-k}$ is invertible, this forces $c_iy^i{\bf1}_\lambda\in M$.
Because $v\not=0$, we have $c_iy^i{\bf1}_\lambda\not=0$ for some $i$ and ${\bf1}_\lambda=\frac{1}{c_i}y^{-i}(c_iy^i{\bf1}_\lambda)\in M$. So $M$ is simple.
\end{proof}

\begin{theo}
Let $M$ be a simple module over $R$. Then one of the following
holds:

(i) There exists a $\lambda\in\mc^*$ such that $M\cong M_\lambda$;

(ii) There exists an $\alpha\in R$ such that $\alpha$ is irreducible
in $B$ and $M\cong (R+B\alpha)/B\alpha$.

\end{theo}
\begin{proof}
If $M$ is $S$-torsion, then there exists $\lambda\in\mc^*$ and $v\in M$ such that $xv=\lambda v$ and $M=\mc[y, y^{-1}]v$, which is isomorphic to $M_\lambda$.

Now assume that $M$ is $S$-torsion-free. By  Lemma \ref{Soc}, there
exists $\alpha\in R$ such that $\alpha$ is irreducible in $B$ and
\beqs M\cong Soc_RB/B\alpha=(R+B\alpha)/B\alpha.\eeqs
\end{proof}


\section{Quantum group $U_q(\sla)$}

Let $q\in\mc^*$ be generic. The  quantum group $U_q(\sla)$ is the
complex unital algebra generated by elements $E, F, K, K^{-1}$ with
relations
\beqs&& KK^{-1}=1, \,KEK^{-1}=q^2E, \,KFK^{-1}=q^{-2}F,\\
&&[E, F]=\frac{K-K^{-1}}{q-q^{-1}}.\eeqs It is well known that the
center of $U_q(\sla)$ is the polynomial algebra $\mc[Z_q]$ in
variable (see \cite{Jan}) \beqs
Z_q=EF+\frac{1}{(q-q^{-1})^2}(q^{-1}K+qK^{-1}).\eeqs

\def\c{{\bf c}}
For any simple $U_q(\sla)$-module $V$, there exists a $\c\in\mc$
such that $Z_q$ acts as $\c\cdot\id_V$. Thus the classification
problem of simple $U_q(\sla)$-modules is equivalent to the
classification problem of simple modules over
$U_{\c}=U_q(\sla)/(Z_q-\c)U_q(\sla)$ for all $\c\in\mc$.

\begin{lemma}
The map $F\mapsto y, K^{\pm1}\mapsto x^{\pm1}, E\mapsto \c y^{-1}-\frac{1}{(q-q^{-1})^2}(q^{-1}K+qK^{-1})y^{-1}$
defines an algebra injection $\phi: U_{\c}\rightarrow R=\mc_\nu[x^{\pm1},y^{\pm1}]$ with $\nu=q^2$.
\end{lemma}
\begin{proof}Straightforward.
\end{proof}

Let $R_\c=\phi(U_\c)$ and $S=\mc[x, x^{-1}]\setminus\{0\}$ as in Section 3.
\begin{lemma}The localization of $R_\c$ with respect to $S$ exists and $S^{-1}R_\c=S^{-1}R$.
\end{lemma}
\begin{proof}
Similar to Lemma \ref{Ore}, $S$ satisfies the left Ore condition in $R_\c$. Then $S^{-1}R_\c$ exists. Since
\beqs y^{-1}=\frac{(q-q^{-1})^2}{ (q-q^{-1})^2\c-(q^{-1}K+qK^{-1}) }\phi(E)\in S^{-1}R_\c,\eeqs
it is easy to see that $S^{-1}R_\c=\mc(x)[y, y^{-1}]=S^{-1}R$.
\end{proof}
As same as in Section 3, let $B=S^{-1}R$.

\begin{theo}
Let $V$ be a simple module over $U_\c$. Then one of the following
holds:

(i) $V$ is a lowest weight module;

(ii) $V$ is a highest weight module;

(iii) $V$ is a simple module of intermediate series;

(iv) There exists an $\alpha\in R_\c$ such that $\alpha$ is
irreducible in $B$ and $V\cong (R_c+B\alpha)/B\alpha$.
\end{theo}
\begin{proof}
If $V$ is $S$-torsion, then there exits a $\lambda\in\mc^*$ and
$v\in V$ such that $K^{\pm1}\cdot v=\lambda^{\pm1}v$. So $V$ is a
simple weight module, and it is known that $V$ is a lowest weight
module, a highest weight module or a simple module of intermediate
series.

Next assume that $V$ is $S$-torsion-free. By  Lemma \ref{Soc}, there
exists an $\alpha\in R_\c$ such that $\alpha$ is irreducible in $B$
and \beqs V\cong Soc_{R_\c}B/B\alpha=(R_\c+B\alpha)/B\alpha.\eeqs
\end{proof}

\section{Examples of new simple $U_q(\sla)$-modules}

For all $f(x), g(x)\in S$, the polynomial $\alpha=f(x)y-g(x)$ is irreducible in $B$. Let $V=(R_\c+B\alpha)/B\alpha$. Then there exists
$v\in V$ such that
\beqs
Fv&=&\frac{g(K)}{f(K)}v,\\
Ev&=&\left(\c-\frac{q^{-1}K+qK^{-1}}{(q-q^{-1})^2}\right)\frac{f(q^{-2}K)}{g(q^{-2}K)}v.
\eeqs

In particular, if $\frac{g(K)}{f(K)}\in\mc$ or $\left(\c-\frac{q^{-1}K+qK^{-1}}{(q-q^{-1})^2}\right)\frac{f(q^{-2}K)}{g(q^{-2}K)}\in\mc$,
$V$ is called a Whittaker module, which is the $q$-analogue of Whittaker module of Lie algebra $\sla$.

\begin{prop}\label{polynomial}
If $V$ is of rank one as free $\mc[K, K^{-1}]$-module, then one of the following holds:

(i) $F{\bf1}=\mu K^n{\bf1}$ and $E{\bf1}=\frac1\mu\left(\c-\frac{q^{-1}K+qK^{-1}}{(q-q^{-1})^2}\right)K^{-n}{\bf1}$;

(ii) $E{\bf1}=\mu K^n{\bf1}$ and $F{\bf1}=\frac1\mu\left(\c-\frac{qK+q^{-1}K^{-1}}{(q-q^{-1})^2}\right)q^{-2n}K^{-n}{\bf1}$;

(iii) $E{\bf1}=\mu K^n(q^{-1}K-x_1){\bf1}$ and $F{\bf1}=\frac{1}{\mu(q-q^{-1})^2}(1-q^{-1}x_2K^{-1})q^{-2n}K^{-n}{\bf1}$.

{\noindent}Where $\mu\in\mc^*$, $n\in\mz$ and $x_1, x_2$ are solutions of $q^{-1}x+qx^{-1}-(q-q^{-1})^2\c=0$.
\end{prop}
\begin{proof}
Suppose that $E{\bf1}=f(K){\bf1}, F{\bf1}=g(K){\bf1}$, then $f(K)g(q^{-2}K)=\c-\frac{qK+q^{-1}K^{-1}}{(q-q^{-1})^2}$. By direct computations, this proposition holds.
\end{proof}

\begin{lemma}Suppose  $V=\mc[K, K^{-1}]{\bf 1}$ such that
$$E{\bf1}=\mu K^n(q^{-1}K-x_1){\bf1},\quad F{\bf1}=\frac{1}{\mu(q-q^{-1})^2}(1-q^{-1}x_2K^{-1})q^{-2n}K^{-n}{\bf1}.$$

(1) If $x_1=q^{-s+1}, x_2=q^{s+1}$ for some positive integer $s$, then $V$ is not a simple module.

(2) If $x_1=-q^{-s+1}, x_2=-q^{s+1}$  for some positive integer $s$, then $V$ is not a simple module.

In particular, in these cases we have ${\bf c}=\pm\frac{q^s+q^{-s}}{(q-q^{-1})^2}$.\end{lemma}
\begin{proof}
We only prove (1), the proof for (2) is similar.

Let $f(K)=\sum_{i=0}^sa_iK^i$ such that $a_0\not=0$ and $a_j=-q^{-2}a_{j-1}\frac{q^{s-j}-q^{j-s}}{q^j-q^{-j}}$ for all $j\geq 1$. Then we have
\beqs \frac{1}{\mu}K^{-n}q^{2s+1}Ef(K){\bf1}&=&(K-q^{s+2})f(K){\bf1},\\
 {\mu}(q-q^{-1})^2K^{n+1}q^{-2s}Ff(K){\bf1}&=&(K-q^{-s})f(K){\bf1}.\eeqs
Thus the subspace $\mc[K, K^{-1}]f(K)$ is a proper submodule over $U_q(\sla)$.
\end{proof}

By the following lemma, the proper submodule $\mc[K, K^{-1}]f(K)$ is simple.

\begin{lemma}
Let $V=\mc[K, K^{-1}]{\bf 1}$ the polynomial module over $U_q(\sla)$ such that
$$E{\bf1}=\mu K^n(q^{-1}K-x_1){\bf1},\quad F{\bf1}=\frac{1}{\mu(q-q^{-1})^2}(1-q^{-1}x_2K^{-1})q^{-2n}K^{-n}{\bf1}.$$
If $x_1\not=\pm q^{-s+1}$ for any positive integer $s$, then $V$ is a simple module.
\end{lemma}
\begin{proof}
We may assume that $n=0$ and $\mu=1$. The proof for general case is very similar.
Then
$E{\bf1}=(q^{-1}K-x_1){\bf1}$ and $F{\bf1}=\frac{1}{(q-q^{-1})^2}(1-q^{-1}x_2K^{-1}){\bf1}$.
By Proposition \ref{polynomial}, $x_1\cdot x_2=q^2$.

For an arbitrary $f(K)=\sum_{i=r}^sa_iK^i\in\mc[K, K^{-1}]$ such that $a_r\cdot a_s\not=0$, we may assume $r=0$
with replacing $f(K)$ by $K^{-r}f(K)$. So $s\geq0$. If $s=0$, then we obtain ${\bf1}$ by multiplying $a_s^{-1}$. Next assume $s>0$. We have
\beqs &&(q^{2s+1}E-K)f(K)\\
&=&(a_{s-1}(q^2-1)-a_sqx_1)K^s\\
&&+(a_{s-2}(q^4-1)-a_{s-1}q^3x_1)K^{s-1}\\
&&\cdots\\
&&+(a_{0}(q^{2s}-1)-a_{1}q^{2s-1}x_1)K\\
&&-a_0q^{2s+1}x_1,
\eeqs
and \beqs &&(q^{-2s}(q-q^{-1})^2KF-K)f(K)\\
&=&(a_{s-1}(q^{-2}-1)-a_sq^{-1}x_2)K^s\\
&&+(a_{s-2}(q^{-4}-1)-a_{s-1}q^{-3}x_2)K^{s-1}\\
&&\cdots\\
&&+(a_{0}(q^{-2s}-1)-a_{1}q^{-2s+1}x_2)K\\
&&-a_0q^{-2s-1}x_2.
\eeqs

{\it Case 1. } $(q^{2s+1}E-K)f(K)$ is not a scalar of $f(K)$. Then
$$g(K):=(a_{s-1}(q^2-1)-a_sqx_1)f(K)-a_s(q^{2s+1}E-K)f(K)\not=0$$
and $g(K)=\sum_{i=0}^{s-1}b_iK^i$ for some constants $b_i\in\mc$.

{\it Case 2. } $(q^{-2s}(q-q^{-1})^2KF-K)f(K)$ is not a scalar of $f(K)$. Similar to Case 1.

{\it Case 3. } Both $(q^{2s+1}E-K)f(K)$ and $(q^{-2s}(q-q^{-1})^2KF-K)f(K)$ are scalars of $f(K)$.

Note that $a_0q^{2s+1}x_1\not=0, a_0q^{-2s-1}x_2\not=0$. By
\beqs a_{0}(q^{2s}-1)-a_{1}q^{2s-1}x_1=-a_1q^{2s+1}x_1,\\
a_{0}(q^{-2s}-1)-a_{1}q^{-2s+1}x_2=-a_1q^{-2s-1}x_2,
\eeqs
we have $a_1\not=0$ and $x_2=x_1q^{2s}$. Hence $x_2=\pm q^{s+1}$ and $x_1=\pm q^{-s+1}$. Contradiction to assumption.

Induction on $s$, we may obtain $s=0$ and then $V$ is a simple module.
\end{proof}

\begin{prop}\label{poly-simp}
Suppose that $V$ is a simple $U_q(\sla)$-module and it is of rank one as free $\mc[K, K^{-1}]$-module. Then one of the following holds:

(i) $F{\bf1}=\mu K^n{\bf1}$ and $E{\bf1}=\frac1\mu\left(\c-\frac{q^{-1}K+qK^{-1}}{(q-q^{-1})^2}\right)K^{-n}{\bf1}$;

(ii) $E{\bf1}=\mu K^n{\bf1}$ and $F{\bf1}=\frac1\mu\left(\c-\frac{qK+q^{-1}K^{-1}}{(q-q^{-1})^2}\right)q^{-2n}K^{-n}{\bf1}$;

(iii) $E{\bf1}=\mu K^n(q^{-1}K-x_1){\bf1}$ and $F{\bf1}=\frac{1}{\mu(q-q^{-1})^2}(1-q^{-1}x_2K^{-1})q^{-2n}K^{-n}{\bf1}$.

{\noindent}Where $\mu\in\mc^*$, $n\in\mz$, $x_1, x_2$ are solutions of $q^{-1}x+qx^{-1}-(q-q^{-1})^2\c=0$ and $x_1\not=\pm q^{-s+1}$ for all positive integer $s$.
\end{prop}
\begin{proof}

For case (i), we have $\frac{1}{\mu}K^{-n}F\cdot\phi(K)=\phi(q^2K)$, then it is easy to know the module is simple.

For case (ii), the proof is similar.

Case (iii) can be obtained by Lemmas 5.2-5.3.
\end{proof}

Note that the Whittaker modules are those modules (i) and (ii) in Proposition \ref{poly-simp} with $n=0$. Consequently, we have the following corollary.
\begin{coro}
The Whittaker modules over $U_q(\sla)$ are simple.
\end{coro}

\def\refname{

\end{document}